\documentclass{article}

\usepackage[top=1in,bottom=1in,left=1.25in,right=1.25in]{geometry}
\usepackage{graphicx}
\usepackage{caption}
\usepackage{subcaption}
\usepackage{float} 
\usepackage{lmodern}
\usepackage[T1]{fontenc}
\usepackage[english,activeacute]{babel}
\usepackage{mathtools}
\usepackage{amsmath,amsthm,amssymb,amsfonts}
\usepackage{mathrsfs}
\usepackage{stmaryrd} 
\usepackage[noadjust]{cite} 
\usepackage{enumerate}
\usepackage{verbatim} 
\usepackage[usenames,dvipsnames,svgnames,table]{xcolor}
\usepackage[utf8]{inputenc} 
\usepackage{soul} 
\usepackage{tikz-cd} 
\usetikzlibrary{babel} 
\usepackage{setspace} 
\usepackage{array, multicol, multirow, makecell}
\usepackage{subfiles}
\usepackage[nottoc]{tocbibind} 
\usepackage{longtable}
\usepackage{tikz}
\usetikzlibrary{positioning}
\usepackage{dynkin-diagrams}
\usepackage{authblk}
\usepackage{abstract}
\usepackage{bbm}

\usepackage[pdftex,backref=page]{hyperref} 

\hypersetup{linktocpage=true,colorlinks=true,citecolor=blue,urlcolor=blue}
\newcommand{\doi}[1]{\href{https://doi.org/#1}{\ttfamily #1}}

\usepackage{import}
\usepackage{xifthen}
\usepackage{pdfpages}
\usepackage{transparent}

\let\OLDthebibliography\thebibliography
\renewcommand\thebibliography[1]{
	\OLDthebibliography{#1}
	\setlength{\parskip}{0pt}
}

\newtheorem{theorem}{Theorem}[section]
\newtheorem{corollary}[theorem]{Corollary}
\newtheorem{proposition}[theorem]{Proposition}

\theoremstyle{definition}
\newtheorem{definition}[theorem]{Definition}

\newtheorem{remark}[theorem]{Remark}

\DeclareMathOperator{\Id}{Id}
\DeclareMathOperator{\spn}{span}

\DeclareMathOperator{\Aut}{Aut}

\DeclareMathOperator{\Ham}{Ham}
\DeclareMathOperator{\Heis}{Heis}

\DeclareMathOperator{\CH}{\mathbb{C}\mathrm{H}}
\DeclareMathOperator{\HH}{\mathbb{H}\mathrm{H}}

\newcommand{\Z}{\mathbb{Z}}

\newcommand{\R}{\mathbb{R}}
\newcommand{\C}{\mathbb{C}}

\newcommand{\calC}{\mathcal{C}}

\newcommand{\calF}{\mathcal{F}}

\newcommand{\calH}{\mathcal{H}}

\newcommand{\calK}{\mathcal{K}}
\newcommand{\calL}{\mathcal{L}}

\newcommand{\calV}{\mathcal{V}}

\newcommand{\frg}{\mathfrak{g}}

\renewcommand{\Re}{\operatorname{Re}}
\renewcommand{\Im}{\operatorname{Im}}

\newcommand{\hor}{\mathrm{H}}
\newcommand{\ver}{\mathrm{V}}
\newcommand{\rmQ}{\mathrm{Q}}

\newcommand{\GL}{\mathrm{GL}}

\newcommand{\SU}{\mathrm{SU}}
\newcommand{\Sp}{\mathrm{Sp}}

\newcommand{\aut}{\mathfrak{aut}}
\newcommand{\heis}{\mathfrak{heis}}

\renewcommand{\sslash}{/\mkern-5mu/}

\newcommand{\pd}[2]{\frac{\partial #1}{\partial #2}}
\renewcommand{\d}{\mathrm{d}}

\DeclarePairedDelimiter{\abs}{\lvert}{\rvert}

\DeclarePairedDelimiter{\escal}{\langle}{\rangle}

\allowdisplaybreaks

\title{\textbf{Symmetries of one-loop deformed q-map spaces}}
\author{Vicente Cortés, Alejandro Gil-García, and Danu Thung}
\affil{\normalsize Fachbereich Mathematik\\
	Universit\"at Hamburg\\
	Bundesstra\ss e 55, 20146 Hamburg, Germany\\
	vicente.cortes@uni-hamburg.de, alejandro.gil.garcia@uni-hamburg.de}
\date{\today}

\begin{document}
	
    \maketitle
	
    \begin{abstract}
    
    Q-map spaces form an important class of quaternionic K\"ahler manifolds of negative scalar curvature. Their one-loop deformations are always inhomogeneous and have been used to construct cohomogeneity one quaternionic K\"ahler manifolds as deformations of homogeneous spaces. Here we study the group of isometries in the deformed case. Our main result is the statement that it always contains a semidirect product of a group of affine transformations of $\mathbb{R}^{n-1}$ with a Heisenberg group of dimension $2n+1$ for a q-map space of dimension $4n$. The affine group and its action on the normal Heisenberg factor in the semidirect product depend on the cubic affine hypersurface which encodes the q-map space.\bigskip
    
    \medskip
    
    
    \end{abstract}

    \clearpage

    \tableofcontents
	
    \section{Introduction}

    Quaternionic K\"ahler manifolds are Riemannian $4n$-manifolds $(n>1)$ whose holonomy group is contained in $\Sp(n)\Sp(1)$, hence they are Einstein. In this paper we are interested in a class of quaternionic K\"ahler manifolds of negative scalar curvature arising from physics, namely \emph{supergravity c-map spaces}. The construction of these spaces have its origin in supergravity. It was shown in \cite{FS90} (see also \cite{Hit09}) that to any projective special K\"ahler manifold $\bar{M}$ of (real) dimension $2n-2$ one can associate a quaternionic K\"ahler manifold $\bar{N}$ of dimension $4n$. This corresponds to dimensional reduction of $\mathcal{N}=2$ supergravity coupled to vector multiplets from 4 to 3 space-time dimensions. It was moreover shown in \cite{RSV06} that the supergravity c-map metric (sometimes referred to as the Ferrara-Sabharwal metric) admits a one-parameter deformation by quaternionic K\"ahler metrics of negative scalar curvature. This metric is known as the \emph{one-loop deformed supergravity c-map metric} (this deformation can be interpreted as a perturbative quantum correction in the string coupling in type IIB string theory).\medskip
    
    A mathematical proof that the (one-loop) deformed supergravity c-map metric is indeed quaternionic K\"ahler of negative scalar curvature can be described as the following two-step process. Given a projective special K\"ahler manifold $\bar M$, one has a conical affine special K\"ahler manifold $M$ on a $\mathbb{C}^*$-bundle over $\bar M$. The cotangent bundle $N=T^*M$ of $M$ is a pseudo-hyper-K\"ahler manifold \cite{CFG89,ACD02} equipped with a rotating circle symmetry \cite{ACM13}, so one can apply the HK/QK correspondence \cite{Hay08,ACM13} to obtain a quaternionic K\"ahler manifold $\bar N$. Finally, one checks that the metric obtained by this method is the deformed supergravity c-map metric \cite{ACDM15}. Summarizing we have the following diagram:

    $$\begin{tikzcd}
M^{2n} \arrow[rrr, "\text{rigid c-map}"]                                          &  &  & N^{4n} \arrow[d, "\text{HK/QK}"] \\
\bar{M}^{2n-2} \arrow[rrr, "\text{supergravity c-map}"] \arrow[u, "\mathbb{C}^*"] &  &  & \bar{N}^{4n}
    \end{tikzcd}$$

    Therefore the deformed supergravity c-map produces a one-parameter family of quaternionic K\"ahler metrics (depending on a real parameter $c\in\R$) from a projective special K\"ahler manifold. Note that the case $c=0$ is the Ferrara-Sabharwal metric. For a fixed projective special K\"ahler manifold, the metrics in the image of the deformed supergravity c-map are locally isometric for different values of $c>0$ \cite{CDS17}. Moreover, any deformed supergravity c-map space is non-locally homogeneous \cite{CGS23}.\medskip
    
    There is a special class of supergravity c-map spaces known as \emph{supergravity q-map spaces}. These spaces arise as the composition of the supergravity r-map and the supergravity c-map. The \emph{supergravity r-map} produces a projective special K\"ahler manifold $\bar{M}$ of (real) dimension $2n-2$ from a projective special real manifold $\calH$ of dimension $n-2$. This construction was introduced in \cite{dWVP92} and it is induced by dimensional reduction of $\mathcal{N}=2$ supergravity theories from 5 to 4 space-time dimensions. Therefore, applying the supergravity q-map to a projective special real manifold $\calH$ of dimension $n-2$, we obtain a quaternionic K\"ahler manifold of dimension $4n$. Since the supergravity c-map admits a one-parameter deformation, so does the supergravity q-map. It is known that for any $c\geq0$, a supergravity q-map space is complete provided that the projective special real manifold is complete \cite{CDS17}. From the resolution of the Alekseevsky conjecture \cite{BL23} we know that all homogeneous quaternionic K\"ahler manifolds of negative scalar curvature are Alekseevsky spaces \cite{Ale75,Cor96}, that is admit a simply transitive solvable Lie group of isometries. Moreover, except $\HH^n$ and $\SU(n,2)/\mathrm{S}(\mathrm{U}(n)\times\mathrm{U}(2))$, all of these spaces are in the image of the supergravity q-map \cite{dWVP92}.\medskip
    
    The purpose of this paper is to determine explicitly a universal Lie group acting effectively and isometrically on any deformed supergravity q-map space $(\bar{N},g_{\bar{N}}^c)$ for $c\geq0$, see Theorem~\ref{thm:effective_action}. For the undeformed case $c=0$ the group was described in \cite[Appendix~A]{CDJL21}, compare with \cite{dWVVP93,dWVP95}.\medskip
    
    The only space that is in the image of the supergravity c-map but is not obtained as a supergravity q-map space is $(\bar{N}=\SU(n,2)/\mathrm{S}(\mathrm{U}(n)\times\mathrm{U}(2)),g_{\bar{N}}^c)$. For this particular case, a Lie group playing a similar role for supergravity c-map spaces (rather than supergravity q-map spaces) acting effectively and isometrically was obtained in \cite{CRT21} by integrating the corresponding Killing vector fields.\medskip
    
    In this paper we directly work on the level of group actions rather than Lie algebra actions. First we recall how the group $\Aut(M)\ltimes\R^{2n}$, $n=\dim_\C M$, acts globally on the rigid c-map space $N=T^*M$, see Propositions~\ref{translations:prop} and \ref{prop:semidirect_Aut-R2n}. Then, we describe how to lift this action to an action of $\Aut(M)\ltimes\Heis_{2n+1}$ on the trivial circle bundle $P$ over $N$ that recovers the known infinitesimal action described in \cite{CST21}, see~\eqref{eq:action_P} and Proposition~\ref{prop:inf_group_action}. Finally, when the conical affine special K\"ahler manifold $M$ is determined by a projective special real manifold $\calH$, we show that the subgroup $$\big((\R_{>0}\times\Aut(\calH))\ltimes\R^{n-1}\big)\ltimes(\Heis_{2n+1}/\calF),$$ where $\calF$ is an infinite cyclic subgroup of the Heisenberg center, acts isometrically and effectively on the quaternionic K\"ahler manifold $\bar{N}\subset P$ viewed as a hypersurface of the circle bundle $P$, yielding our main result. Here $\Aut(\calH)$ denotes the automorphism group of the projective special real manifold $\calH$. The cyclic group $\calF$ included to ensure effectiveness can be removed by considering the universal covering of the quaternionic K\"ahler manifold, which amounts to replacing the circle bundle $P$ by an $\mathbb{R}$-bundle.\medskip
    
    \textbf{Notation:} We will use the Einstein summation convention throughout the text for the computations in coordinates.
    
    \subsubsection*{Acknowledgements}
    
    V.\ C.\ and A.\ G.\ are supported by the German Science Foundation (DFG) under Germany’s Excellence Strategy – EXC 2121 “Quantum Universe” – 390833306.
    
    \section{Symmetries of rigid c-map spaces}\label{sec:sym_rigid}
 
    In this section we recall some relevant background material. In particular, we recall the definition and some properties of conical affine special K\"ahler manifolds, as well as the rigid c-map construction, which assigns a pseudo-hyper-K\"ahler manifold with a rotating Killing vector field to any conical affine special K\"ahler manifold. Moreover, we review how to construct Hamiltonian automorphisms of the rigid c-map structure.
    
    \subsection{Special K\"ahler manifolds}
	
	\begin{definition}
		An \emph{affine special K\"ahler (ASK) manifold} $(M,g,J,\nabla)$ is a (pseudo-) K\"ahler manifold $(M,g,J)$ endowed with a flat torsion-free connection $\nabla$ such that $\nabla\omega=0$ and $\d^\nabla J=0$, where $\omega=gJ:=g(J\cdot,\cdot)$ denotes the K\"ahler form.
	\end{definition}
 
        \begin{definition}
		A \emph{conical affine special K\"ahler (CASK) manifold} $(M,g,J,\nabla,\xi)$ is an ASK manifold $(M,g,J,\nabla)$ endowed with a complete vector field $\xi$, called the \emph{Euler vector field}, such that \begin{itemize}
			\itemsep 0em
			\item $g$ is negative-definite on $\spn\{\xi,J\xi\}$ and positive-definite on its orthogonal complement,
			\item $D\xi=\nabla\xi=\Id$, where $D$ is the Levi-Civita connection of $g$.
		\end{itemize} Moreover, $M$ is endowed with a principal $\C^*$-action generated by $\xi$ and $J\xi$.
	\end{definition}

	We collect in the following proposition some well-known facts about the interaction between the vector fields $\xi,J\xi$ and the affine special K\"ahler structure (see e.g.\ \cite[Proposition~3]{CM09} and \cite[page~652]{CHM12} for a proof).
	
	\begin{proposition}\label{prop:basis_CASK}
		Let $(M,g,J,\nabla,\xi)$ be a CASK manifold. Then \begin{enumerate}[\normalfont(a)]
			\itemsep 0em
			\item The vector field $\xi$ is holomorphic and homothetic: $\mathcal L_\xi g = 2g$.
			\item The vector field $J\xi$ is holomorphic and Killing.
			\item The function $f=\frac{1}{2}g(\xi,\xi)$ is a K\"ahler potential for $g$ and a Hamiltonian function for $J\xi$, where our convention is $\d f=-\omega(J\xi,\cdot)$.
		\end{enumerate}
	\end{proposition}
 
    Given a CASK manifold $(M,g,J,\nabla,\xi)$, the quotient space $\bar{M}:=M/\C^*$ inherits a (positive-definite) K\"ahler metric $\bar{g}$ and $(\bar{M},\bar{g})$ is called a \emph{projective special K\"ahler (PSK) manifold}. The metric is obtained by K\"ahler reduction exploiting the fact that $J\xi$ is a Hamiltonian Killing vector field. Indeed, $$\bar{M}=M\sslash S^1=f^{-1}(-\tfrac{1}{2})/S^1$$ is the K\"ahler quotient of the CASK manifold $M$ by the $S^1$-action generated by $J\xi$.\medskip
    
    Following \cite{CST21} we introduce the natural notion of symmetry in this setting.
	
	\begin{definition}\textcolor{white}{}
		\begin{itemize}
			\itemsep 0em
			\item An \emph{automorphism} of a CASK manifold $(M,g,J,\nabla,\xi)$ is a diffeomorphism of $M$ which preserves $g$, $J$, $\nabla$ and $\xi$.
			\item An \emph{automorphism} of a PSK manifold $\bar{M}=M\sslash S^1$ is a diffeomorphism of $\bar{M}$ induced by an automorphism of the CASK manifold $M$.
		\end{itemize}
	\end{definition}

	The corresponding groups of automorphisms are denoted by $\Aut(M)$ and $\Aut(\bar M)$, respectively. At the infinitesimal level we have the following notion.
	
	\begin{definition}\textcolor{white}{}
		\begin{itemize}
			\itemsep 0em
			\item An \emph{infinitesimal automorphism} of a CASK manifold $(M,g,J,\nabla,\xi)$ is a vector field $X\in\Gamma(TM)$ such that its local flow preserves the CASK data on $M$. The Lie algebra of such vector fields is denoted by $\aut(M)$.
			\item An \emph{infinitesimal automorphism} of a PSK manifold is a vector field $\bar{X}$ induced by an infinitesimal automorphism of the corresponding CASK manifold (which always projects since it commutes with $\xi$ and $J\xi$). The corresponding Lie algebra is denoted by $\aut(\bar{M})$.
		\end{itemize}
	\end{definition}

	It was shown in \cite[Proposition~2.18]{CST21} that $\aut(M)$ and $\aut(\bar{M})$ are isomorphic when the Levi-Civita connection $D$ and the special connection $\nabla$ are not equal.
	
	\subsection{The rigid c-map}
	
	Let us consider a CASK manifold $(M,g,J,\nabla,\xi)$. We will discuss the natural geometric structure that exists on its cotangent bundle $N=T^*M$. For this, it will be useful to recall some basic facts about vector bundles (see e.g.\ \cite{MS22}).\medskip
	
	Let $\pi:E\longrightarrow M$ be a vector bundle over a manifold $M$. We can pull $E$ back to a bundle over its total space: $\pi^*E\longrightarrow E$. This bundle always admits a tautological section $\Phi\in\Gamma(\pi^*E)$, which assigns to the point $e\in E$ the value $e$. In the case where $E=T^*M$, this is precisely (one interpretation of) the tautological one-form $\lambda$. The tangent vectors to the fibers of $E$ determine a canonical vertical distribution $T^\ver E\subset TE$, which is moreover canonically isomorphic to $\pi^*E$. The corresponding isomorphism is denoted by $$\calV:\pi^*E\longrightarrow T^\ver E.$$
	
	In the case where $E=T^*M$, local coordinates $\{q^j\}$ on $M$ induce canonical coordinates $\{q^j,p_j\}$ on $T^*M$ and the isomorphism $\calV:\pi^*(T^*M)\longrightarrow T^\ver(T^*M)$ is implemented by mapping $\pi^*(\d q^j)$ to $\pd{}{p_j}$.\medskip
	
	Now assume that $E$ comes equipped with some connection $\nabla$. Then we may use the pullback connection on $\pi^*E$ to compute $(\pi^*\nabla)\Phi$. The assignment $X\longmapsto(\pi^*\nabla)_X\Phi$, where $X\in\Gamma(TE)$, then provides a left inverse to $\calV$. Thus, a vector field $X$ is determined by $(\pi^*\nabla)_X\Phi\in\Gamma(\pi^*E)$ and $\pi_*(X)\in\Gamma(\pi^*(TM))$. This is just another way of phrasing the fact that a connection on $E$ induces a splitting $TE\cong\pi^*(TM)\oplus\pi^*E$. In particular, for $E=T^*M$ we obtain the splitting $T(T^*M)\cong\pi^*(TM)\oplus\pi^*(T^*M)$. With respect to this splitting, we may define the following tensor fields: \begin{equation}\label{eq:rigid_c-map_structure}
		g_N=\begin{pmatrix}
			g&0\\
			0&g^{-1}
		\end{pmatrix},\qquad I_1=\begin{pmatrix}
			J&0\\
			0&J^*
		\end{pmatrix},\qquad I_2=\begin{pmatrix}
			0&-\omega^{-1}\\
			\omega&0
		\end{pmatrix},\qquad I_3=I_1I_2.
	\end{equation}

	It is well-known that these tensor fields define a pseudo-hyper-K\"ahler structure on $N=T^*M$ if and only if $(M,g,J,\nabla)$ is an ASK manifold \cite{CFG89,ACD02}. Furthermore, the conical structure gives rise to the vector field $Z:=-\widetilde{J\xi}$, where the tilde denotes the horizontal lifting with respect to $\nabla$ (i.e.\ $(\pi^*\nabla)_Z\lambda=0$ where $\lambda$ is the tautological one-form). This vector field is Killing and $\omega_1$-Hamiltonian, but satisfies $\calL_Z\omega_2=\omega_3$ and $\calL_Z\omega_3=-\omega_2$ \cite[Proposition~2]{ACM13}. It is therefore called a \emph{rotating} Killing field, and the circle action it generates is called rotating as well. The Hamiltonian function for $Z$ with respect to $\omega_1$ is denoted by $f_Z^c$, and is given by $f_Z^c=-\frac{1}{2}g_N(Z,Z)-\frac{1}{2}c$, $c\in\R$.\medskip
	
	Taken together, the above hyper-K\"ahler structure with rotating circle action is referred to as the \emph{rigid c-map structure} on $N=T^*M$ and $(N,g_N,I_1,I_2,I_3)$ as the \emph{rigid c-map space}.
	
	\begin{definition}
		A diffeomorphism $\varphi:N\longrightarrow N$ is called an \emph{automorphism of the rigid c-map structure}, or equivalently of the hyper-K\"ahler structure with rotating circle action, if it preserves $g_N$, $I_1$, $I_2$, $I_3$ and $f_Z^c$.
	\end{definition}

	Note that an automorphism in the above sense automatically commutes with the rotating circle action. The group of all automorphisms of the rigid c-map structure is denoted by $\Aut_{S^1}(N)$. The subgroup of $\omega_1$-Hamiltonian automorphisms is denoted by $\Ham_{S^1}(N)$.\medskip
	
	Given a hyper-K\"ahler structure with rotating circle action, there is a canonical, closed two-form $\omega_\hor:=\omega_1+\d\iota_Zg_N$ of type (1,1) with respect to each $I_k$ \cite[Lemma~2.7]{CST22} that plays a central role in the following. Note that any vector field which is $\omega_1$-Hamiltonian and preserves $Z$ and $g_N$ is automatically $\omega_\hor$-Hamiltonian as well. This applies, in particular, to the rotating Killing field $Z$, whose Hamiltonian with respect to $\omega_\hor$ we denote by $f_\hor^c$, and is given by $f_\hor^c=f_Z^c+g_N(Z,Z)=\frac{1}{2}g_N(Z,Z)-\frac{1}{2}c$. In the following, we will focus on Hamiltonian functions with respect to $\omega_\hor$ rather than $\omega_1$.\medskip
	
	There are two important sources of Hamiltonian automorphisms of the rigid c-map structure: canonical lifts of CASK automorphisms and translations in the fibers.\medskip
	
	For completeness of the exposition we mention that in the case of the rigid c-map the pre-symplectic form $\omega_\hor$ is in fact symplectic.
	
	\begin{proposition}
		Let $(N,g_N,I_1,I_2,I_3)$ be a rigid c-map space. Then the canonical two-form $\omega_\hor=\omega_1+\d\iota_Zg_N$ associated with the rotating Killing field $Z$ is a symplectic structure.
	\end{proposition}

	\begin{proof}
		Since the one-form $\iota_Zg_N$ is the pull-back of $\alpha=-\iota_{J\xi}g$ we can calculate its differential as $\pi^*\d\alpha$, where $\d\alpha$ can be expressed as twice the skew-symmetric part of the Levi-Civita covariant derivative $$D\alpha=-g(D(J\xi),\cdot)=-g(J\cdot,\cdot)=-\omega.$$
		
		Since this is skew-symmetric, we see that $\d\iota_Zg_N=\pi^*\d\alpha=-2\pi^*\omega$. This shows that \begin{equation}\label{oH:eq}
			\omega_\hor=\begin{pmatrix}
				-\omega&0\\
				0&\omega^{-1}
			\end{pmatrix}.
		\end{equation}
	\end{proof}

	Incidentally, the above expression for $\omega_\hor$ shows that the endomorphism $$I_\hor:=g_N^{-1}\omega_\hor=\begin{pmatrix}
		-J&0\\
		0&J^*
	\end{pmatrix}$$ is a $g_N$-skew-symmetric almost complex structure, which obviously commutes with $I_1$, $I_2$ and $I_3$. In this way we recover the statement that $\omega_\hor$ is of type $(1,1)$ for the three complex structures.

	\begin{corollary}
		Any rigid c-map space $(N,g_N,I_1,I_2,I_3)$ carries a canonical almost K\"ahler structure $(g_N,I_\hor)$.
	\end{corollary}
 
    \subsection{Canonical lifts}
    
    Let $\Aut(M)$ be the group of CASK automorphisms of the CASK manifold $M$. Canonically lifting to $N=T^*M$, we obtain a group of $\omega_\hor$-Hamiltonian automorphisms of the rigid c-map structure, as proven in \cite{CST21}. It was remarked in \cite[Proposition~2.10]{CRT21} that this group even admits an equivariant moment map $\mu:N\longrightarrow\frg^*$ with respect to $\omega_\hor$. We recall the result in Proposition~\ref{prop:moment_map} below adding details and fixing notation.\medskip
	
	Let $X$ denote a vector field on $M$ generating a one-parameter family of automorphisms of the CASK structure, and $Y$ its canonical lift to $N$. Then, with respect to the splitting $T(T^*M)=T^\hor N\oplus T^\ver N\cong\pi^*(TM)\oplus\pi^*(T^*M)$, we may write $$Y=\pi^*X-\lambda\circ(\pi^*\nabla)(\pi^*X).$$
	
	Here $\lambda:\xi\longmapsto\lambda_\xi$ is the tautological section of the vector bundle $\pi^*(T^*M)$, defined as $$\lambda_\xi(v)=\xi(v),\qquad\xi\in N_p=T_p^*M,\,v\in(\pi^*(TM))_\xi=T_pM,\,\pi(\xi)=p,$$ where $\pi^*X=X\circ\pi\in\Gamma(\pi^*(TM))$ and $\lambda\circ(\pi^*\nabla)(\pi^*X)\in\Gamma(\pi^*(T^*M))\subset\Gamma (T^*N)$ is the one-form sending a vector field $A\in\Gamma(TN)$ to the smooth function $\lambda((\pi^*\nabla)_A(\pi^*X))$, and where we have identified $\pi^*(T^*M)=(\pi^*(TM))^*=(T^\hor N)^*=(T^\ver N)^0\subset T^*N$.\medskip
	
	Note that \begin{equation}\label{pi*nabla:eq}
		(\pi^*\nabla)_Y\lambda=-\lambda\circ(\pi^*\nabla)(\pi^*X).
	\end{equation}

	Indeed, for a vector field $Y\in\Gamma(TN)$ and $\lambda\in\Gamma(\pi^*(T^*M))$ the tautological section, $(\pi^*\nabla)_Y\lambda$ gives us the vertical component of $Y$, which is precisely $-\lambda\circ(\pi^*\nabla)(\pi^*X)$.\medskip
	
	In local coordinates $\{q^i,p_i\}$ on $N$ induced by local $\nabla$-affine coordinates $\{q^i\}$ on $M$, the tautological section $\lambda$ on $N$ is given by $\lambda(\pd{}{q^i})=p_i$ and the vector field $Y\in\Gamma(TN)$ by $$Y=X^j\pd{}{q^j}-p_i\pd{X^i}{q^j}\pd{}{p_j}$$ in terms of the components $\{X^i\}$ of the vector field $X$ in the local coordinates $\{q^i\}$.
	
	\begin{proposition}\label{prop:moment_map}
		Let $X$ be an infinitesimal CASK automorphism and $Y$ its canonical lift as explained above. Then $$\mu_Y=\frac{1}{2}\left(g_N(Z,Y)+\pi^*\omega^{-1}(\lambda\circ(\pi^*\nabla)(\pi^*X),\lambda)\right)$$ is a $\omega_\hor$-Hamiltonian function for $Y$, where $Z=-\widetilde{J\xi}$. This assignment determines an equivariant (co)moment map $\mu:\aut(M)\longrightarrow\calC^\infty(N)$, $X\longmapsto\mu_X:=\mu_Y$, for the action of $\aut(M)$ on $N=T^*M$.
	\end{proposition}

	\begin{proof}
		We have to show that $\iota_Y\omega_\hor=-\d\mu_Y$. Recall that $\omega_\hor$ is given by \eqref{oH:eq} with respect to the splitting $T(T^*M)\cong\pi^*(TM)\oplus\pi^*(T^*M)$. Thus we have $$\iota_Y\omega_\hor=-\pi^*(\iota_X\omega)+\iota_{(\pi^*\nabla)_Y\lambda}\pi^*\omega^{-1}.$$
		
		Let us compute these two terms. First notice that, by Proposition~\ref{prop:basis_CASK}~(a) we have that 
		$\mathcal L_\xi\omega=2\omega$ and, hence $$2\iota_X\omega=\iota_X\mathcal L_\xi\omega =\iota_X\d(\iota_\xi\omega)=\mathcal L_X(\iota_\xi\omega)-\d(\iota_X\iota_\xi\omega)=\d(g(-J\xi,X)),$$ 
		since $\mathcal L_X(\iota_\xi\omega)=0$ (recall that $X$ is an infinitesimal CASK automorphism). So we see that \begin{equation}\label{iotaXomega:eq}
			\pi^*(\iota_X\omega)=\frac{1}{2}\d(g_N(Z,Y)).
		\end{equation}
	
		To compute the second term, let $A$ be an arbitrary section of $\pi^*(T^*M)\cong T^\ver(T^*M)$. Let $\{q^j\}$ be special real coordinates on $M$ and $\{q^j,p_j\}$ the corresponding canonical coordinates on $N$. With respect to these, we can write $A=A_i\d q^i$ and, using $(\pi^*\nabla)_Y\lambda=-\lambda\circ(\pi^*\nabla)(\pi^*X)=-p_i\pd{X^i}{q^j}\d q^j$, we get $$\pi^*\omega^{-1}((\pi^*\nabla)_Y\lambda,A)=-\omega^{jk}p_i\pd{X^i}{q^j}A_k.$$
		
		Thus, $\iota_{(\pi^*\nabla)_Y\lambda}\pi^*\omega^{-1}$ is the one-form on $T^*M$ which vanishes when applied to a horizontal vector field and evaluates on a vertical vector field $A=A_i\pd{}{p_i}$ (corresponding to the $A$ consider earlier by the canonical isomorphism $\pi^*(T^*M)\cong T^\ver(T^*M)$) as above. This means that we can write it, in local coordinates, as \begin{equation}\label{iotapi*nabla:eq}
			\iota_{(\pi^*\nabla)_Y\lambda}\pi^*\omega^{-1}=-\omega^{jk}p_i\pd{X^i}{q^j}\d p_k.
		\end{equation}
	
		Now, we compute the differential of our proposed moment map: $$\d\mu_Y=\frac{1}{2}\left(\d(g_N(Z,Y))+\d(\pi^*\omega^{-1}(\lambda\circ(\pi^*\nabla)(\pi^*X),\lambda))\right),$$ the second term of which can be computed (using \eqref{pi*nabla:eq} and \eqref{iotapi*nabla:eq}) in local coordinates as follows: \begin{align*}
			\frac{1}{2}\d(\pi^*\omega^{-1}(\lambda\circ(\pi^*\nabla)(\pi^*X),\lambda))&=\frac{1}{2}\d\left(\omega^{jk}p_ip_k\pd{X^i}{q^j}\right)=\frac{1}{2}\omega^{jk}\pd{X^i}{q^j}(p_i\d p_k+p_k\d p_i)\\
			&=\omega^{jk}p_i\pd{X^i}{q^j}\d p_k=-\iota_{(\pi^*\nabla)_Y\lambda}\pi^*\omega^{-1},
		\end{align*} where, in passing to the second line, we used that both $\omega^{jk}$ and $\pd{X^i}{q^j}$ are anti-symmetric in their indices; the latter fact is just the equation $\mathcal L_X \omega=0$ in ($\nabla$-affine) coordinates.\medskip
  
        This computation, together with \eqref{iotaXomega:eq}, concludes the proof that $\mu_Y$ is indeed a Hamiltonian function for $Y$.\medskip
		
		To show that the map $\mu$ is equivariant, let $X_1,X_2\in\aut(M)$ and $Y_1,Y_2\in\Gamma(TN)$ their corresponding canonical lifts. Note that the canonical lift of $[X_1,X_2]$ is precisely $[Y_1,Y_2]$. Using this and the fact that the moment map is constructed in a canonical way, we obtain $$\calL_{Y_1}(\mu_{Y_2})=\frac{1}{2}\left(g_N(Z,[Y_1,Y_2])+\pi^*\omega^{-1}(\lambda\circ(\pi^*\nabla)(\pi^*([X_1,X_2])),\lambda)\right)=\mu_{[Y_1,Y_2]}.$$
	\end{proof}

	\begin{remark}\textcolor{white}{}
		\begin{itemize}
			\itemsep 0em
			\item Note that the $\omega_\hor$-Hamiltonian function $\mu_Y$ given by Proposition~\ref{prop:moment_map} is the only moment map that is homogeneous, i.e.\ $\calL_\Xi\mu_Y=2\mu_Y$, where $\Xi\in\Gamma(TN)$ is the sum of the $\nabla$-horizontal lifted Euler vector field $\xi$ with the fiberwise Euler (or position) vector field of $T^*M$. In canonical coordinates associated to (conical) special real coordinates, $\Xi$ takes the form $\Xi=q^j\pd{}{q^j}+p_j\pd{}{p_j}$.
			\item Note also that the equivariance of the moment map $\mu:\mathfrak{aut}(M)\longrightarrow\calC^\infty(N)$ (without assuming homogeneity) fixes it uniquely up to adding a linear form $c:\mathfrak{aut}(M)\longrightarrow\R$, $X\longmapsto c_X$, invariant under the coadjoint representation. The space of such forms is trivial if and only if the Lie algebra is perfect, that is, it coincides with its derived ideal. This is the case, in particular, for semisimple Lie algebras.
		\end{itemize}
	\end{remark}

	Thus, we have a canonical action of $\Aut(M)$ on $N$ which casts $\Aut(M)$ as a subgroup of $\Ham_{S^1}(N)$, with a canonical choice of equivariant moment map.
	
	\subsection{Translations in the fibers}
	
	One of the crucial features of the rigid c-map metric on $N$ is that it is semi-flat. This means that it is foliated by half-dimensional flat submanifolds. These are just the fibers of $\pi:N=T^*M\longrightarrow M$, and the reason they are flat is that the metric in each fiber is constant with respect to the affine structure induced by the vector space structure of the fiber, as can be seen directly from \eqref{eq:rigid_c-map_structure}.\medskip
	
	An important consequence is the following result.
	
	\begin{proposition}\label{translations:prop}
		The cotangent bundle $N=T^*M$ carries locally an $\omega_\hor$-Hamiltonian action of the group $\R^{2n}$ by automorphisms of the rigid c-map structure, which preserves the fibers and acts on them by translations. If the holonomy group of the special connection $\nabla$ on $M$ is trivial (in particular, if $M$ is simply connected), then the action is global.
	\end{proposition}

	\begin{proof}
		We first give a description of the local action. Choosing special real coordinates $\{q^j\}$ on $M$ and using the associated canonical coordinates $\{q^j,p_j\}$ on $N$, the action of $\R^{2n}$ is generated by the locally defined vector fields $\pd{}{p_j}$. A quick look at \eqref{eq:rigid_c-map_structure} reveals that the full hyper-K\"ahler structure is preserved by these vector fields. Moreover, they commute with the rotating circle action as well, since $Z=-\widetilde{J\xi}$ only depends on the coordinates $\{q^j\}$ on the base $M$. This implies that they also preserve the $\omega_1$-Hamiltonian function $f_Z^c=-\frac{1}{2}g_N(Z,Z)-\frac{1}{2}c$. Finally, let us give a local Hamiltonian with respect to $\omega_\hor$. Writing \eqref{oH:eq} in the above coordinates, we have $$\omega_\hor=\frac{1}{2}(-\omega_{ij}\d q^i\wedge\d q^j+\omega^{ij}\d p_i\wedge \d p_j),$$ where each $\omega_{ij}$ is constant (and hence so are the components $\omega^{ij}$ of the inverse matrix), and we are omitting pullbacks when no confusion can arise.\medskip
  
        Hence, we find $$\iota_{\pd{}{p_k}}\omega_\hor=\omega^{kj}\d p_j=-\d(-\omega^{kj}p_j).$$
		
		We can thus assign the local Hamiltonian function \begin{equation}\label{eq:mu_R2n}
			\mu_{\pd{}{p_k}}=-\omega^{kj}p_j
		\end{equation} to $\pd{}{p_k}$. More generally, $\mu$ assigns to every vertical vector field $v=v_k\pd{}{p_k}$ with constant coefficients the function $\mu_v=-\omega^{kj}v_kp_j$. The vector fields $\pd{}{p_k}$ generate an action of the group $\R^{2n}$ with the claimed properties. In the local coordinates $\{q^j,p_j\}$ a vector $v=(v_j)\in\R^{2n}$ acts by $(q,p)\longmapsto(q,p+v)$.\medskip
	
		To define a global moment map $\mu:\R^{2n}\longrightarrow\calC^\infty(N)$ and a global group action of $\R^{2n}$ on $N$ it suffices to have a global frame of the vertical bundle $T^\ver N\cong\pi^*(T^*M)$ parallel with respect to the connection $\pi^*\nabla$. Such a frame exists if and only if the holonomy of the (flat) special connection is trivial.
	\end{proof}

	As a comment on the last step of the proof of Proposition~\ref{translations:prop} we note that locally (over the preimage $\pi^{-1}(U)$ of a suitable open set $U\subset M$) a parallel frame of $T^\ver N\cong\pi^*(T^*M)$ can be chosen of the form $\pi^*(\d q^i)$, where $\{q^i\}$ are local $\nabla$-affine coordinates.
 
 \subsection{The semidirect product}\label{subsection:semidirect_product}
	
	Under the assumption that the holonomy group of $\nabla$ is trivial, we have constructed two subgroups of $\Ham_{S^1}(N)$ and the next thing to do is to study how they interact. In this section, we check that the generators of the two groups combine into a semidirect product. Let us start by emphasizing that the group $\Gamma_\nabla(T^*M)\cong\R^{2n}$ of $\nabla$-parallel sections acts by addition on $N=T^*M$, that is $$\alpha\cdot\beta=\alpha(p)+\beta,$$ for all $\alpha\in\Gamma_\nabla(T^*M)$, $\beta\in T_p^*M$, $p\in M$. The group $\Aut(M)$ acts naturally on $N=T^*M$: $$h\cdot\beta=h_*\beta=(h^{-1})^*\beta,$$ for all $h\in\Aut(M)$, $\beta\in T_p^*M$, $p\in M$.
	
	\begin{proposition}\label{prop:semidirect_Aut-R2n}
		The subgroups $\R^{2n}$ and $\Aut(M)$ of $\Ham_{S^1}(N)$ generate a group $G\subset\Ham_{S^1}(N)$ which is a semidirect product $\Aut(M)\ltimes\R^{2n}$.
	\end{proposition}

	\begin{proof}
		It is clear that the two subgroups have trivial intersection, as $\Aut(M)$ lifts a subgroup of diffeomorphisms of $M$ whereas $\R^{2n}$ preserves each fiber of $N=T^*M$. The action of $\Aut(M)$ is linear on the fibers and preserves the space of $\nabla$-parallel one-forms. We check explicitly that the elements $h\in\Aut(M)$ normalize the group of translations: $$(h\cdot\alpha\cdot h^{-1})\cdot\beta=h_*(\alpha(h^{-1}(p))+h^*\beta)=(h_*\alpha)(p)+\beta=(h_*\alpha)\cdot\beta,$$ for all $\alpha\in\Gamma_\nabla(T^*M)$, $\beta\in T_p^*M$, $p\in M$. This proves that $h\cdot\alpha\cdot h^{-1} =h_*\alpha\in\Gamma_\nabla(T^*M)$ for all $\alpha\in\Gamma_\nabla(T^*M)$.
        \end{proof}

	Infinitesimally, we can describe the semidirect product structure in terms of structure constants if we choose a basis. Thus, let $\{Y_\alpha\}$ be infinitesimal generators of the action of $\Aut(M)$ on $N$, obtained by canonically lifting generators $\{X_\alpha\}\subset\aut(M)$. With respect to canonical coordinates induced by (conical) special real coordinates on $M$, we can express $Y_\alpha$ as $$Y_\alpha=X_\alpha^j\pd{}{q^j}-p_i\pd{X_\alpha^i}{q^j}\pd{}{p_j},$$ where the component functions $X_\alpha^j$ of $X_\alpha$ are linear functions. The generators of the $\R^{2n}$-action are the vectors $\pd{}{p_k}$. The structure constants are now easily computed: $$\left[Y_\alpha,\pd{}{p_k}\right]=\pd{X_\alpha^k}{q^j}\pd{}{p_j}.$$
	
	Note that the coefficients multiplying $\pd{}{p_j}$ on the right-hand side are indeed constants, since any $X\in\aut(M)$ is $\nabla$-affine (i.e.\ $\calL_X\nabla=0$). Together with the structure constants of $\Aut(M)$ this determines the structure of $\aut(M)\ltimes\R^{2n}$.
	
	\section{Symmetries under the HK/QK correspondence}
	
	Now that we have a completely explicit description of the group $\Aut(M)\ltimes\R^{2n}\subset\Ham_{S^1}(N)$, the next step is to transfer the group action to the quaternionic K\"ahler manifold $\bar{N}$. Up to a so-called elementary modification explained below, this is essentially done by first lifting to $P=N\times S^1$ and subsequently studying the induced action on $\bar{N}$, which we realize as a submanifold of $P$. We will continue to assume that the holonomy group of the flat special connection $\nabla$ is trivial to ensure the global $\R^{2n}$-action.

    \subsection{The HK/QK correspondence as a twist}
    
    The HK/QK correspondence \cite{Hay08,ACM13} produces a quaternionic K\"ahler manifold from a (pseudo-)hyper-K\"ahler manifold equipped with a rotating circle action. This correspondence was studied in \cite{MS15} using the twist formalism introduced by Swann in \cite{Swa10}. We briefly recall here this construction (see \cite{Swa10} for details).\medskip

    Let $N$ be a manifold equipped with an $S^1$-action generated by a vector field $Z\in\Gamma(TN)$, and let $\pi_N:P\longrightarrow N$ be a principal circle bundle over $N$ with connection one-form $\eta\in\Omega^1(P)$ and curvature $\omega\in\Omega^2(N)$ (i.e.\ $\d\eta=\pi_N^*\omega$). We want to lift the vector field $Z$ to a vector field $Z_P\in\Gamma(TP)$ so that it preserves the connection $\eta$, i.e.\ $\calL_{Z_P}\eta=0$, and it commutes with the principal circle action, i.e.\ $[Z_P,X_P]=0$, where $X_P\in\Gamma(TP)$ generates the principal circle action. It turns out that such a lift exists if and only if $\iota_Z\omega=-\d f_Z^c$ for some function $f_Z^c\in\calC^\infty(N)$. The lift is given by $$Z_P=\tilde{Z}+\pi_N^*f_Z^cX_P,$$ where $\tilde{Z}$ denotes the horizontal lift with respect to $\eta$. The triple $(Z,\omega,f_Z^c)$ with the above properties is called \emph{twist data}. A manifold $N$ equipped with twist data produces a new smooth\footnote{We will always assume that both $Z$ and $Z_P$ generate free and proper actions so that the quotient space $\bar N$ is a smooth manifold.} manifold $\bar N:=P/\escal{Z_P}$ called the \emph{twist} of $N$ with respect to the twist data $(Z,\omega,f_Z^c)$. We then have a double fibration structure on $P$:
    
    $$\begin{tikzcd}
    & P \arrow[ld, "\pi_N"'] \arrow[rd, "\pi_{\bar N}"] \arrow[<-, out=65, in=25, loop, "Z_P"] \arrow[out=155, in=115, loop, "X_P"] & \\
    N \arrow[loop left, "Z"] \arrow[rr, <->, "\mathrm{twist}"', dashed] && \bar{N} \arrow[<-, loop right, "Z_{\bar N}"]
    \end{tikzcd}$$

    Recall that $X_P$ generates the principal circle action with respect to the projection $\pi_N:P\longrightarrow N$ and note that $Z_P$ plays the same role for the projection $\pi_{\bar N}:P\longrightarrow\bar N$. Note also that the twist construction produces a circle action on $\bar N$ generated by the vector field $Z_{\bar N}:=\d\pi_{\bar N}(X_P)\in\Gamma(T\bar N)$.\medskip

    Now let $(N,g_N,I_1,I_2,I_3)$ be a pseudo-hyper-K\"ahler manifold equipped with a rotating Killing vector field $Z$ generating a circle action. To obtain a quaternionic K\"ahler manifold from the twist construction, one first needs to ``deform'' the pseudo-hyper-K\"ahler metric $g_N$ to another metric $g_\hor^c$ and then apply the twist construction to the manifold $(N,g_\hor^c)$ with the appropriate twist data. This is the idea behind an elementary deformation. Therefore, let $\omega_1:=g_N(I_1\cdot,\cdot)$ and $f_Z^c$ a nowhere vanishing function such that $\iota_Z\omega_1=-\d f_Z^c$. Define $$\omega_\hor:=\omega_1+\d\iota_Zg_N,\quad f_\hor^c:=f_Z^c+g_N(Z,Z),\quad g_\hor^c:=\frac{1}{f_Z^c}g_N|_{(\mathbb{H}Z)^\perp}+\frac{f_\hor^c}{(f_Z^c)^2}g_N|_{\mathbb{H}Z},$$ where $\mathbb{H}Z:=\spn\{Z,I_1Z,I_2Z,I_3Z\}$. One can check that $(Z,\omega_\hor,f_\hor^c)$ is a twist data for $N$. It was shown then in \cite{MS15} that the twist of $(N,g_\hor^c)$ with respect to the twist data $(Z,\omega_\hor,f_\hor^c)$ is a quaternionic K\"ahler manifold $(\bar N,g_{\bar N}^c)$. The upshot is that we can use the twist formalism to study the geometry of the quaternionic K\"ahler manifold $\bar N$ purely in terms of the geometry of $N$ and the twist data. This will be the approach in the following subsections.
    
    \subsection{Infinitesimal description}
	
    The infinitesimal description of the transfer of symmetries under the HK/QK correspondence appears explained in \cite{CST21}: one performs an (elementary) modification and then twists. More precisely, let $\frg\subset\mathfrak{ham}_{S^1}(N)$ be a subalgebra with moment map $\mu:N\longrightarrow\frg^*$ with respect to $\omega_\hor$. Denote the Hamiltonian function corresponding to $V\in\frg$ by $\mu_V$. Then, the first step consists in modifying $V$ to $$V_\hor:=V-\frac{\mu_V}{f_\hor^c}Z\in\Gamma(TN).$$
	
	We will sometimes refer to $V_\hor$ as the elementary deformation of $V$. The second step is twisting $V_\hor$ to a vector field $\operatorname{tw}(V_H)\in\Gamma(T\bar{N})$, which we will denote by $V_{\rmQ}$. Twisting is done by lifting $V_\hor$ horizontally (with respect to a given connection $\eta$ whose curvature is $\omega_\hor$) to the trivial circle bundle $P=N\times S^1$ and then projecting down to $\bar{N}$. In other words $$\operatorname{tw}(V_\hor)=V_{\rmQ}=\d\pi_{\bar{N}}(\tilde{V}_\hor)=\d\pi_{\bar{N}}\left(\tilde{V}-\frac{\mu_V}{f_\hor^c}\tilde{Z}\right),$$ where we have denoted the $\eta$-horizontal lift to $P$ by a tilde.\medskip
	
	This procedure gives rise to an injective, linear map $\varphi_\mu:\frg\longrightarrow\aut_{S^1}(\bar{N})$, dependent on the choice of moment map $\mu$. Here, $\aut_{S^1}(\bar{N})$ denotes the space of Killing fields of the quaternionic K\"ahler manifold $(\bar{N},g_{\bar{N}}^c)$, $c\geq0$, which commute with the canonical, isometric circle action on $\bar{N}$ generated by $Z_{\bar{N}}=\operatorname{tw}(-\frac{1}{f_\hor^c}Z)$.\medskip
	
	It is shown in \cite{CST21} that this linear map is a homomorphism of Lie algebras if and only if the moment map $\mu$ is equivariant. Indeed, choosing a basis $\{V_i\}$ of $\frg$ with corresponding structure constants $\{c_{ij}^k\}$, the equivariance condition can be expressed as $$\omega_\hor(V_i,V_j)-c_{ij}^k\mu_{V_k}=0.$$
	
	The left-hand side is precisely what measures the failure of $\varphi_\mu$ to be a homomorphism, according to \cite[Theorem~3.8]{CST21}.\medskip
	
	In the case at hand, we are considering $\frg\cong\aut(M)\ltimes\R^{2n}$, and we have a canonical moment map. Since this canonical moment map is equivariant for the action of the subgroup $\Aut(M)$ by Proposition~\ref{prop:moment_map}, we obtain a subalgebra of $\aut_{S^1}(\bar{N})$ isomorphic to $\aut(M)$. However, the canonical choice of moment map we gave for $\R^{2n}$ in \eqref{eq:mu_R2n} is not equivariant. Indeed $$\omega_\hor\left(\pd{}{p_i},\pd{}{p_j}\right)=\omega^{ij}\neq0,$$ while $\R^{2n}$ is of course abelian and therefore has vanishing structure constants. Following \cite[Theorem~3.8]{CST21}, this implies that $\R^{2n}$ gives rise to a subalgebra of $\aut_{S^1}(\bar{N})$ which is a one-dimensional central extension of $\R^{2n}$ by $Z_{\bar{N}}$, and whose non-trivial brackets are given by the coefficients of $\omega^{-1}$. In other words, $$\left[\left(\pd{}{p_i}\right)^{\rmQ},\left(\pd{}{p_j}\right)^{\rmQ}\right]=\omega^{ij}Z_{\bar{N}}.$$
	
	Since $\omega^{-1}$ is the natural symplectic form on the fibers of $N=T^*M$, the central extension in question is nothing but the Heisenberg algebra $\heis_{2n+1}$.\medskip
	
	We thus obtain two algebras of Killing fields, isomorphic to $\aut(M)$ and $\heis_{2n+1}$, respectively. Together, they generate an algebra which is once again a semidirect product $\aut(M)\ltimes\heis_{2n+1}$. To see that they indeed form a semidirect product, it suffices to show that $[\aut(M),\heis_{2n+1}]\subset\heis_{2n+1}$. To check this, consider $Y_\alpha^{\rmQ}\in\aut(M)\subset\mathfrak{aut}_{S^1}(\bar{N})$ and $(\pd{}{p_k})^{\rmQ}\in\R^{2n}\subset\heis_{2n+1}$ (we already know that $Z_{\bar{N}}$ is central in the full algebra), where the vector fields $Y_\alpha$ were introduced in Section~\ref{subsection:semidirect_product}. Then, again by \cite[Theorem~3.8]{CST21}, we have $$\left[Y_\alpha^\rmQ,\left(\pd{}{p_k}\right)^{\rmQ}\right]=\pd{X_\alpha^k}{q^j}\left(\pd{}{p_j}\right)^\rmQ+\operatorname{tw}\left(\omega_\hor\left(Y_\alpha,\pd{}{p_k}\right)-\pd{X_\alpha^k}{q^j}\mu_{\pd{}{p_j}}\right)Z_{\bar{N}},$$ where $X_\alpha\in\Gamma(TM)$ lifts to $Y_\alpha$. Now we use local coordinates to compute $$\omega_\hor\left(Y_\alpha,\pd{}{p_k}\right)=-\pd{X_\alpha^k}{q^j}\omega^{j\ell}p_\ell=\pd{X_\alpha^k}{q^j}\mu_{\pd{}{p_j}}$$ and conclude that $$\left[Y_\alpha^\rmQ,\left(\pd{}{p_k}\right)^{\rmQ}\right]=\pd{X_\alpha^k}{q^j}\left(\pd{}{p_j}\right)^\rmQ,$$ so $\heis_{2n+1}$ is an ideal inside the Lie algebra $\frg$ generated by $\aut(M)$ and $\heis_{2n+1}$ and we have $\frg\cong\aut(M)\ltimes\heis_{2n+1}$. Summarizing, we have obtained the following result.
 
    \begin{proposition}
        There exists a subalgebra of $\aut_{S^1}(\bar{N})$ which is isomorphic to the semidirect product $\aut(M)\ltimes\heis_{2n+1}$.
    \end{proposition}
 
    \subsection{Global description}
 
    We would like to construct a global counterpart of the above construction, by lifting the action of the group $\Aut(M)\ltimes\R^{2n}$ on $N$ to an action of the semidirect product $\Aut(M)\ltimes\Heis_{2n+1}$ on $P$ and then projecting to $\bar{N}$ directly, without passing to the generating vector fields and having to integrate them as intermediate steps. Here the product in the Heisenberg group $\Heis_{2n+1}$ is realized on the product manifold $\R^{2n}\times \R$, where $\R^{2n}\cong\Gamma_\nabla(T^*M)$, $$(\alpha_1 , \tau_1 )\cdot (\alpha_2 , \tau_2 ) = (\alpha_1 + \alpha_2 ,\tau_1+\tau_2 + \tfrac12 \omega^{-1}(\alpha_1, \alpha_2 )),$$ where the constant function $\omega^{-1}(\alpha_1, \alpha_2 )$ is identified with a number.\medskip
    
    Let us consider $h\in\Aut(M)$ and $(\alpha,\tau)\in\Heis_{2n+1}$. Then we have the following group action of $\Aut(M)\ltimes \Heis_{2n+1}$ on the trivial circle bundle $P$ over $N=T^*M$: \begin{equation}\label{eq:action_P}
		\begin{split}
			h\cdot(\beta,s)&=(h_*\beta,s),\\
			(\alpha,\tau)\cdot(\beta,s)&=(\alpha(p)+\beta,s+[\tau+\tfrac{1}{2}\omega^{-1}(\alpha(p),\beta)]),
		\end{split}
	\end{equation} where $s\in S^1=\R/2\pi \Z$, $\beta\in T_p^*M$, $p\in M$ and $[r] = r \pmod{2\pi}$.\medskip

	This action covers the action of $\Aut(M)\ltimes\R^{2n}$ on $N$ by means of the quotient homomorphism $\Heis_{2n+1}\longrightarrow\R^{2n}\cong\Heis_{2n+1}/\R$. Moreover, it induces the infinitesimal action on $\bar N$ from the previous section.
 
    \begin{proposition}\label{prop:inf_group_action}
		The infinitesimal group action on $P$ corresponding to \eqref{eq:action_P} descends to the action of $\aut(M)\ltimes\heis_{2n+1}$ on $\bar N$ as described above.
    \end{proposition}

    \begin{proof}
		Let $V\in\mathfrak{ham}_{S^1}(N)\subset\Gamma(TN)$ and $\tilde V$ its $\eta$-horizontal lift to $P$. We define a lift of $V$ to $P$ by $$\hat{V}:=\tilde{V}+\varphi X_P\in\Gamma(TP),$$ where $\varphi\in\calC^\infty(P)$ and $X_P=\pd{}{s}$ generates the principal circle action. We require that the lift $\hat V$ preserves the connection $\eta$. The condition $\calL_{\hat V}\eta=0$ then implies that $\varphi=\pi_N^*f$ for some $\omega_\hor$-Hamiltonian function $f\in\calC^\infty(N)$ for $V$. Thus $\hat V=\tilde V+\pi_N^*fX_P$. Such lift automatically commutes with $X_P$. We will omit the pullbacks in the notation from now on.\medskip
		
		If we choose $f=\mu_V$, the canonical Hamiltonian, we recover the procedure described in the previous section. Indeed, note that $$\tilde{V}_\hor-\hat V=-\frac{\mu_V}{f_\hor^c}Z_P,$$ where $Z_P=\tilde{Z}+f_\hor^cX_P$. Since $\d\pi_{\bar N}(Z_P)=0$, by definition of $\bar N$, we get $V_\rmQ=\d\pi_{\bar N}(\hat V)$. This means that we obtain the same Killing vector field on $\bar N$ projecting the lift $\hat V$ or twisting the elementary deformation $V_\hor$.\medskip
		
		Now let us work out explicitly what the infinitesimal lift $\hat{V}$ looks like. Since $P$ is trivial, we may regard any vector field on $N$ as a vector field on $P$ which is horizontal with respect to the product structure (or equivalently with respect to the trivial connection $\d s$, with $s$ a local coordinate on $S^1$). Thus, we may write the $\eta$-horizontal lift $\tilde{V}$ as $$\tilde{V}=V-\eta(V)X_P,$$ and similarly \begin{equation}\label{eq:V_hat}
            \hat{V}=\tilde{V}+\mu_VX_P=V-(\eta(V)-\mu_V)X_P.
        \end{equation}
		
		A canonical choice of connection $\eta$ with curvature $\omega_\hor$ is given by $$\eta=\d s+\frac{1}{2}\iota_\Xi\omega_\hor,$$ where $\Xi\in\Gamma(TN)$ is the vector field expressed in coordinates by $\Xi=q^j\pd{}{q^j}+p_j\pd{}{p_j}$, thus $$\eta=\d s+\frac{1}{2}(-\omega_{ij}q^i\d q^j+\omega^{ij}p_i\d p_j).$$
		
		Recall that a canonically lifted automorphism of $M$ takes the form $$Y=X^j\pd{}{p_j}-\pd{X^i}{q^j}p_i\pd{}{p_j}\in\mathfrak{ham}_{S^1}(N)$$ with canonical Hamiltonian function (see Proposition~\ref{prop:moment_map}) $$\mu_Y=\frac{1}{2}\left(-\omega_{ij}q^iX^j+\omega^{jk}p_ip_k\pd{X^i}{q^j}\right).$$
		
		Using these expressions, we find $\eta(Y)=\frac{1}{2}\iota_Y\iota_\Xi\omega_\hor=\mu_Y$. The upshot is that $\hat{Y}=Y$ (see \eqref{eq:V_hat}), i.e.\ the lifted action of $\Aut(M)$ to $P=N\times S^1$ is trivial on the $S^1$-factor and it corresponds to the action of $\Aut(M)$ described in the first equation of \eqref{eq:action_P}.\medskip
		
		Finally, we consider the lift of the group $\R^{2n}$. For an element of its Lie algebra $v\in\R^{2n}$ we have the local expression $v=v_k\pd{}{p_k}$ and corresponding moment map $\mu_v=-\omega^{kj}v_kp_j$. This time, we find $\eta(v)=\frac{1}{2}\mu_v$, and consequently $\hat{v}=v+\frac{1}{2}\mu_vX_P$. These vector fields do not induce an action of $\R^{2n}$ since they no longer commute. Indeed, we have \begin{align*}
			[\hat{v},\hat{w}]&=\frac{1}{2}(v(\mu_w)-w(\mu_v))X_P=\frac{1}{2}(\d\mu_w(v)-\d\mu_v(w))X_P\\
			&=\frac{1}{2}(-(\iota_w\omega_\hor)(v)+(\iota_v\omega_\hor)(w))=\omega_\hor(v,w)X_P,
		\end{align*} or, in local coordinates, $[\hat{v},\hat{w}]=\omega^{ij}v_iw_j\pd{}{s}$. Since $\omega^{ij}$ is constant and $\pd{}{s}$ is central, the conclusion is that we are now dealing with an infinitesimal action of a one-dimensional central extension of $\R^{2n}$ whose non-trivial commutators are given by a symplectic form, i.e.\ a Heisenberg algebra $\heis_{2n+1}$. Integrating, we obtain the action of $\Heis_{2n+1}$ described in the second line of \eqref{eq:action_P}.
    \end{proof}
    
    It is possible to describe the quaternionic K\"ahler manifold $\bar N$ as a submanifold of the circle bundle $P$. For that we define the following tensor fields on $P$\footnote{We are using a convention different from \cite{ACDM15}. The dictionary is given by $Z=-\frac{1}{2}Z^{\mathrm{ACDM}}$, $f_Z^c=-\frac{1}{2}f^{\mathrm{ACDM}}$ and $f_\hor^c=-\frac{1}{2}f_1^{\mathrm{ACDM}}$.}: $$g_P:=-\frac{1}{f_\hor^c}\eta^2+\pi_N^*g_N,\quad\theta_0^P:=\d f_Z^c,\quad\theta_1^P:=\eta-\iota_Zg_N,\quad\theta_2^P:=-\iota_Z\omega_3,\quad\theta_3^P:=\iota_Z\omega_2.$$
    
    With them, define the $Z_P$-invariant tensor field$$\tilde{g}_P:=g_P+\frac{1}{f_Z^c}\sum_{j=0}^3(\theta_j^P)^2,$$ where recall that $Z_P=\tilde{Z}+f_\hor^c X_P$, and consider $$g_{\bar N}^c:=\frac{1}{4\abs{f_Z^c}}\tilde{g}_P|_{\bar N},$$ where \begin{equation}\label{eq:barN_in_P}
        \bar N:=\{\arg(X^0)=0\}\subset P=N\times S^1
    \end{equation} is a codimension one submanifold of $P$ which is transversal to the vector field $Z_P$ and $(X^0,\ldots,X^{n-1})$ are special holomorphic coordinates of the CASK manifold $M$. Then, by \cite[Theorems~2 and 5]{ACDM15}, $(\bar N,g_{\bar N}^c)$ is precisely the one-loop deformed c-map space.\medskip
    
    To state Theorem~\ref{thm:effective_action}, we have to focus on a particular class of CASK manifolds, namely those coming from projective special real manifolds $\calH$. We now briefly describe them. We follow the notation and conventions of \cite{CT22}. For more details, see for instance \cite{CHM12,CDJL21}.
    
    \begin{definition}
		A \emph{projective special real (PSR) manifold} is a Riemannian manifold $(\calH,g_{\calH})$ such that $\calH\subset\R^{n-1}$ is a hypersurface and there is a homogeneous cubic polynomial $h:\R^{n-1}\longrightarrow\R$ satisfying: \begin{itemize}
			\itemsep 0em
			\item $\calH\subset\{t\in\R^{n-1}\mid h(t)=1\}$.
			\item $g_\calH=-\partial^2h|_{T\calH\times T\calH}$.
		\end{itemize}
    \end{definition}
    
    We denote the coordinates of $\R^{n-1}$ by $t^a$. In particular we write the cubic polynomial $h$ as $$h(t^a)=\frac{1}{6}k_{abc}t^at^bt^c,$$ where the coefficients $k_{acb}\in\R$ are symmetric in the indices.\medskip
 
    Let $(\calH,g_\calH)$ be a PSR manifold defined by the real cubic polynomial $h$, and let $U:=\R_{>0}\cdot\calH\subset\R^{n-1}\setminus\{0\}$. We define $\bar{M}:=\R^{n-1}+iU\subset\C^{n-1}$ with the canonical holomorphic structure, where the global holomorphic coordinates are given by $z^a:=b^a+it^a\in\R^{n-1}+iU$. On $\bar{M}$ we consider the metric $$\bar{g}=\frac{\partial^2\calK}{\partial z^a\partial\bar{z}^b}\d z^a\d\bar{z}^b,$$ where $\calK:=-\log K(t)$ and $K(t):=8h(t)=\frac{4}{3}k_{abc}t^at^bt^c$. Then it can be shown that $(\bar{M},\bar{g})$ is a PSK manifold \cite{CHM12}. In fact, one can find the following explicit expression for a PSK metric coming from a PSR manifold: $$\bar g=-\frac{1}{4}\frac{\partial^2\log h(t)}{\partial t^a\partial t^b}(\d b^a\d b^b+\d t^a\d t^b)=\left(-\frac{k_{abc}t^c}{4h(t)}+\frac{k_{acd}k_{bef}t^ct^dt^et^f}{(4h(t))^2}\right)(\d b^a\d b^b+\d t^a\d t^b).$$

    \begin{definition}
        The construction of the PSK manifold $(\bar M,\bar g)$ from the PSR manifold $(\calH,g_\calH)$ as explained above is known as the \emph{supergravity r-map}.
    \end{definition}
    
    Given a PSK manifold $\bar M$ in the image of the supergravity r-map, its corresponding CASK manifold is defined via the CASK domain $(M,\mathfrak{F})$, where $M\subset\C^n$ is given by $$M:=\{(X^0,\ldots,X^{n-1})=X^0\cdot(1,z)\in\C^n\mid X^0\in\C^*,z\in\bar{M}\}$$ and $$\mathfrak{F}(X)=-\frac{h(X^1,\ldots,X^{n-1})}{X^0}=-\frac{1}{6}k_{abc}\frac{X^aX^bX^c}{X^0}.$$
    
    We know recall what is a CASK domain and how to obtain a CASK manifold from it.
    
    \begin{definition}
		A \emph{CASK domain} is a pair $(M,\mathfrak{F})$ such that: \begin{itemize}
			\itemsep 0em
			\item $M\subset\C^n\setminus\{0\}$ is a $\C^*$-invariant domain (with respect to the usual $\C^*$-action on $\C^n\setminus\{0\}$ by multiplication).
			\item $\mathfrak{F}:M\longrightarrow\C$ is a holomorphic function, homogeneous of degree 2 with respect to the $\C^*$-action, called the \emph{holomorphic prepotential}.
			\item With respect to the natural holomorphic coordinates $(X^0,\ldots,X^{n-1})$ of $M$, the matrix $$(\tau_{ij}):=\left(\frac{\partial^2\mathfrak{F}}{\partial X^i\partial X^j}\right)$$ satisfies that $\Im(\tau_{ij})$ has signature $(n-1,1)$ and $\Im(\tau_{ij})X^i\bar{X}^j<0$ for $X\in M$.
		\end{itemize}
	\end{definition}

	From this data, we obtain a CASK manifold $(M,g,\omega,\nabla,\xi)$ by $$g=\Im(\tau_{ij})\d X^i\d\bar{X}^j,\quad\omega=\frac{i}{2}\Im(\tau_{ij})\d X^i\wedge\d\bar{X}^j,\quad\xi=X^i\pd{}{X^i}+\bar{X}^i\pd{}{\bar{X}^i}.$$
 
    The flat connection $\nabla$ is defined such that \begin{equation}\label{eq:flat_frame}
	    \d x^i:=\Re(\d X^i)\quad\text{and}\quad\d y_i:=-\Re(\d(\pd{\mathfrak{F}}{X^i}))
     \end{equation} is a flat frame of $T^*M$.\medskip
     
     Note that for a CASK manifold $M$ determined by $\calH$ as explained above, the functions $x^i:=\Re(X^i)$ and $y_i:=-\Re(\pd{\mathfrak{F}}{X^i})$ are globally defined and their differentials \eqref{eq:flat_frame} give rise to a parallel frame of $T^*M$, hence $\mathrm{Hol}(\nabla)$ is trivial and therefore the action of $\R^{2n}$ given by Proposition~\ref{translations:prop} is global.\medskip
     
     In this situation, where the CASK manifold $M$ is determined by a PSR manifold $\calH$, a group of isometries preserving the CASK structure was described in \cite[Appendix~A]{CDJL21}. More precisely, the group we are considering is $$\mathrm{Aff}_{\calH}(\R^{n-1}):=(\R_{>0}\times\Aut(\calH))\ltimes\R^{n-1}\hookrightarrow\Aut(M)\subset\Sp(\R^{2n}),$$ where $$\Aut(\calH):=\{A\in\GL(n-1,\R)\mid A\calH=\calH\}$$ and the arrow is a certain embedding \cite[Proposition~23]{CDJL21}.
    
    \begin{theorem}\label{thm:effective_action}
        Let $M$ be a CASK manifold determined by the PSR manifold $\calH$ as explained above. Then the group $\mathrm{Aff}_{\calH}(\R^{n-1})\ltimes(\Heis_{2n+1}/\calF)$, where $\calF$ is an infinite cyclic subgroup of the Heisenberg center, acts effectively and isometrically on $(\bar N,g_{\bar N}^c)$ for $c\geq0$.
    \end{theorem}

    \begin{proof}
        From the explicit description of the action of $\mathrm{Aff}_{\calH}(\R^{n-1})$ on $M$ given in \cite[Appendix~A]{CDJL21} we see that the function $X^0\in\calC^\infty(P)$ changes under this action only by a real positive factor. It follows that the submanifold $\bar N\subset P$ given by \eqref{eq:barN_in_P} is invariant under the action of $\mathrm{Aff}_{\calH}(\R^{n-1})\ltimes\Heis_{2n+1}$, since the action of $\Heis_{2n+1}$ preserves the fibers of $\pi\circ\pi_N:P\longrightarrow M$, where $\pi:N=T^*M\longrightarrow M$ and $\pi_N:P\longrightarrow N$. Hence the group action on $P$ described in \eqref{eq:action_P} restricts to $\bar N$. From the explicit description, it is also easy to check that the action of $\mathrm{Aff}_{\calH}(\R^{n-1})$ is effective on $\bar N\subset P$.\medskip
        
        Finally, the action of $\Heis_{2n+1}$ is not quite effective, since its center acts by translations along the $S^1$-factor of $P$. To obtain an effective action we need to divide out the infinite cyclic subgroup of the center, whose elements correspond to shifting $s\in S^1=\R/2\pi\Z$ by $2\pi k$, $k\in\Z$.\medskip

        The group $\mathrm{Aff}_{\calH}(\R^{n-1})\ltimes(\Heis_{2n+1}/\calF)$ acts by isometries on $\bar N$ by construction.
    \end{proof}

    \begin{remark}
        In \cite[Theorem~3.16]{CRT21} a similar result as in Theorem~\ref{thm:effective_action} is obtained in the case of the CASK manifold $M$ determined by the PSK manifold $\bar M=\CH^{n-1}$ with the transitive action of $\Aut(M)=\SU(1,n-1)$. This example does not belong to the above series of spaces determined by a PSR manifold $\calH$ since $\CH^{n-1}$ is not in the image of the supergravity r-map. They obtain that $\SU(1,n-1)\ltimes(\Heis_{2n+1}/\calF')$ acts effectively and isometrically on $(\bar N'/\calF',g_{\bar N}^c)$, where $\bar{N}'\cong\R^{4n}$ is the universal covering of our $\bar N$, and $\calF'$ is trivial if $c=0$ and infinite cyclic for $c>0$. In the latter case $\bar{N}'/\calF'=\bar N$.
    \end{remark}
    
    \bibliographystyle{amsalpha}
    \bibliography{biblio}

\providecommand{\bysame}{\leavevmode\hbox to3em{\hrulefill}\thinspace}
\providecommand{\MR}{\relax\ifhmode\unskip\space\fi MR }
\providecommand{\MRhref}[2]{%
  \href{http://www.ams.org/mathscinet-getitem?mr=#1}{#2}
}
\providecommand{\href}[2]{#2}
\begin{thebibliography}{dWVVP93}

\bibitem[ACD02]{ACD02}
D.~V. Alekseevsky, V.~Cortés, and C.~Devchand, \emph{Special complex
  manifolds}, J. Geom. Phys. \textbf{42} (2002), 85--105,
  \doi{10.1016/S0393-0440(01)00078-X}.

\bibitem[ACDM15]{ACDM15}
D.~V. Alekseevsky, V.~Cortés, M.~Dyckmanns, and T.~Mohaupt, \emph{Quaternionic
  {K}ähler metrics associated with special {K}ähler manifolds}, J. Geom.
  Phys. \textbf{92} (2015), 271--287, \doi{10.1016/j.geomphys.2014.12.012}.

\bibitem[ACM13]{ACM13}
D.~V. Alekseevsky, V.~Cortés, and T.~Mohaupt, \emph{Conification of {K}ähler
  and hyper-{K}ähler manifolds}, Commun. Math. Phys. \textbf{324} (2013),
  no.~2, 637--655, \doi{10.1007/s00220-013-1812-0}.

\bibitem[Ale75]{Ale75}
D.~V. Alekseevsky, \emph{Classification of quaternionic spaces with a
  transitive solvable group of motions}, Izv. Akad. Nauk SSSR Ser. Mat.
  \textbf{39} (1975), no.~2, 315--362, \doi{10.1070/im1975v009n02abeh001479}.

\bibitem[BL23]{BL23}
C.~Böhm and R.~A. Lafuente, \emph{Non-compact {E}instein manifolds with
  symmetry}, J. Amer. Math. Soc. \textbf{36} (2023), no.~3, 591--651,
  \doi{10.1090/jams/1022}.

\bibitem[CDJL21]{CDJL21}
V.~Cortés, M.~Dyckmanns, M.~Jüngling, and D.~Lindemann, \emph{A class of
  cubic hypersurfaces and quaternionic {K}ähler manifolds of co-homogeneity
  one}, Asian J. Math. \textbf{25} (2021), no.~1, 1--30,
  \doi{10.4310/AJM.2021.v25.n1.a1}.

\bibitem[CDS17]{CDS17}
V.~Cortés, M.~Dyckmanns, and S.~Suhr, \emph{Completeness of projective special
  {K}ähler and quaternionic {K}ähler manifolds}, Special metrics and group
  actions in geometry, Springer INdAM series, vol.~23, Springer, Cham, 2017,
  \doi{10.1007/978-3-319-67519-0\_4}, pp.~81--106.

\bibitem[CFG89]{CFG89}
S.~Cecotti, S.~Ferrara, and L.~Girardello, \emph{Geometry of type {II}
  superstrings and the moduli of superconformal field theories}, Int. J. Mod.
  Phys. A \textbf{4} (1989), no.~10, 2475--2529,
  \doi{10.1142/S0217751X89000972}.

\bibitem[CGS23]{CGS23}
V.~Cortés, A.~{Gil-García}, and A.~Saha, \emph{A class of locally
  inhomogeneous complete quaternionic {K}ähler manifolds}, Commun. Math. Phys.
  \textbf{403} (2023), no.~3, 1611--1626, \doi{10.1007/s00220-023-04830-6}.

\bibitem[CHM12]{CHM12}
V.~Cortés, X.~Han, and T.~Mohaupt, \emph{Completeness in supergravity
  constructions}, Commun. Math. Phys. \textbf{311} (2012), no.~1, 191--213,
  \doi{10.1007/s00220-012-1443-x}.

\bibitem[CM09]{CM09}
V.~Cortés and T.~Mohaupt, \emph{Special geometry of {E}uclidean supersymmetry
  {III}: the local r-map, instantons and black holes}, J. High Energy Phys.
  \textbf{2009} (2009), no.~7, 066, \doi{10.1088/1126-6708/2009/07/066}.

\bibitem[Cor96]{Cor96}
V.~Cortés, \emph{Alekseevskian spaces}, Differential Geom. Appl. \textbf{6}
  (1996), no.~2, 129--168, \doi{https://doi.org/10.1016/0926-2245(96)89146-7}.

\bibitem[CRT21]{CRT21}
V.~Cortés, M.~Röser, and D.~Thung, \emph{Complete quaternionic {K}ähler
  manifolds with finite volume ends}, \url{https://arxiv.org/abs/2105.00727},
  2021.

\bibitem[CST21]{CST21}
V.~Cortés, A.~Saha, and D.~Thung, \emph{Symmetries of quaternionic {K}ähler
  manifolds with {$S^1$}-symmetry}, Trans. London Math. Soc. \textbf{8} (2021),
  no.~1, 95--119, \doi{10.1112/tlm3.12026}.

\bibitem[CST22]{CST22}
\bysame, \emph{Curvature of quaternionic {K}ähler manifolds with
  {$S^1$}-symmetry}, Manuscripta Math. \textbf{168} (2022), 35--64,
  \doi{10.1007/s00229-021-01294-7}.

\bibitem[CT22]{CT22}
V.~Cortés and I.~Tulli, \emph{S-duality and the universal isometries of q-map
  spaces}, Commun. Math. Phys. \textbf{394} (2022), no.~2, 833--885,
  \doi{10.1007/s00220-022-04413-x}.

\bibitem[dWVP92]{dWVP92}
B.~de~Wit and A.~Van~Proeyen, \emph{Special geometry, cubic polynomials and
  homogeneous quaternionic spaces}, Commun. Math. Phys. \textbf{149} (1992),
  no.~2, 307--333, \doi{10.1007/BF02097627}.

\bibitem[dWVP95]{dWVP95}
\bysame, \emph{Isometries of special manifolds}, Proceedings of the Meeting on
  Quaternionic Structures in Mathematics and Physics, Trieste, September 1994
  (1995), 92--118, \url{https://www.emis.de/proceedings/QSMP94/}.

\bibitem[dWVVP93]{dWVVP93}
B.~de~Wit, F.~Vanderseypen, and A.~Van~Proeyen, \emph{Symmetry structure of
  special geometries}, Nucl. Phys. B \textbf{400} (1993), no.~1-3, 463--521,
  \doi{https://doi.org/10.1016/0550-3213(93)90413-J}.

\bibitem[FS90]{FS90}
S.~Ferrara and S.~Sabharwal, \emph{Quaternionic manifolds for type {II}
  superstring vacua of {C}alabi-{Y}au spaces}, Nucl. Phys. B \textbf{332}
  (1990), no.~2, 317--332, \doi{10.1016/0550-3213(90)90097-W}.

\bibitem[Hay08]{Hay08}
A.~Haydys, \emph{Hyper{K}ähler and quaternionic {K}ähler manifolds with
  {$S^1$}-symmetries}, J. Geom. Phys. \textbf{58} (2008), no.~3, 293--306,
  \doi{10.1016/j.geomphys.2007.11.004}.

\bibitem[Hit09]{Hit09}
N.~Hitchin, \emph{Quaternionic {K}ähler moduli spaces}, Riemannian topology
  and geometric structures on manifolds, Progress in mathematics, vol. 271,
  Birkhäuser Boston, 2009, \doi{10.1007/978-0-8176-4743-8\_3}, pp.~49--61.

\bibitem[MS15]{MS15}
O.~Macia and A.~Swann, \emph{Twist geometry of the c-map}, Commun. Math. Phys.
  \textbf{336} (2015), no.~3, 1329--1357, \doi{10.1007/s00220-015-2314-z}.

\bibitem[MS22]{MS22}
M.~Mantegazza and A.~Saha, \emph{The c-map as a functor on certain variations
  of {H}odge structure}, Geom. Dedicata \textbf{216} (2022), no.~3, 32 (42 pp),
  \doi{10.1007/s10711-022-00692-9}.

\bibitem[RSV06]{RSV06}
D.~{Robles Llana}, F.~Saueressig, and S.~Vandoren, \emph{String loop corrected
  hypermultiplet moduli spaces}, J. High Energy Phys. \textbf{2006} (2006),
  no.~3, 081, \doi{10.1088/1126-6708/2006/03/081}.

\bibitem[Swa10]{Swa10}
A.~Swann, \emph{Twisting {H}ermitian and hypercomplex geometries}, Duke Math.
  J. \textbf{155} (2010), no.~2, 403--431, \doi{10.1215/00127094-2010-059}.

\end{thebibliography}
    
\end{document}